\documentclass[12pt]{article}
\usepackage[numbers,sort&compress]{natbib}
\usepackage{amssymb,amsmath,mathrsfs,amsfonts,amsthm}
\usepackage{enumerate}
\usepackage{paralist}
\usepackage{color}
\usepackage{epsfig}
\usepackage{graphicx}
\usepackage{setspace}
\usepackage{appendix}

\newtheorem{dl}{Theorem}[section]
\newtheorem{yl}{Lemma}[section]

\newtheorem{zy}{Remark}[section]

\newtheorem{dy}{Definition}[section]

\newtheorem{mytheorem}{Theorem}
\allowdisplaybreaks

\newcommand{\rmnum}[1]{\romannumeral #1}
\newcommand{\Rmnum}[1]{\expandafter\@slowromancap\romannumeral #1@}

\begin{document}

\let\oldsection\section
\renewcommand\section{\setcounter{equation}{0}\oldsection}
\renewcommand\thesection{\arabic{section}}
\renewcommand\theequation{\thesection.\arabic{equation}}

\title{\textbf{Refined estimates of the propagation speed in porous medium equation
of combustion type{%
}}}
\renewcommand{\thefootnote}{\fnsymbol{footnote}}

\author{ Suying Liu\footnote{School of Mathematics and Statistics, Northwestern Polytechnical University, Shaanxi, Xi'an, 710129, China;
 suyingliu@nwpu.edu.cn}\ \   and\ \   Fan Wu\footnote{Chern Institute of Mathematics, Nankai University, Tianjin 300071, China; 1120220015@mail.nankai
 .edu.cn}
}
\date{}
\maketitle
\begin{abstract}
We are concerned with the Cauchy problem $u_{t}=(u^{m})_{xx}+f(u)$, where the nonliearity $f(u)$ is of combustion type and the initial data is compactly supported. In \cite{lou2024convergence}, among other things, the authors prove that by considering a multiple of a given initial data, there is a critical value such that  the corresponding  transition solution  spreads at the asymptotic speed $2y_{0}\sqrt{t}[1+o(1)]\ \text{as} \ t\rightarrow\infty$, while the lower order term $o(1)$ remains unknown. In this paper, for a family of functions of combustion type, we refine the estimates of the asymptotic speed of the transition solution and provide a precise characterization of the lower  order term $o(1)$.  Our result also reveals that  there is no unified characterization of the lower order term for general combustion type functions $f$.

\end{abstract}

\section{Introduction}
\renewcommand{\themytheorem}{\Alph{mytheorem}}
In this paper, we study the porous medium equation with reactions:
\begin{equation}
\left\{
\begin{split}
&u_{t}=(u^{m})_{xx}+f(u),\qquad \, x\in \mathbb{R}, \ t>0,\\
&u(x,0)=u_{0}(x),\ \ \ \quad\qquad x\in \mathbb{R},\\
\end{split}\label{1.1}
\right.
\end{equation}
where $m>1$ is a constant, $u_0$ has compact support and $f$ satisfies
\begin{align}
f(x)\in C^{2}([0,\infty))\ \text{with Lipschitz number}\ K, \ f(0)=0\ \text{and}\ f(x)<0\ \text{for}\ x>1.\label{ff}
\end{align}

In recent decades, problem (\ref{1.1}) has been extensively studied in the literature. It has been proved in \cite{de1991travelling,aronson1982stabilization,sacks1983initial,vazquez2007porous} that under condition (\ref{ff}), (\ref{1.1}) admits a unique globally defined solution $u(x,t)$. Moreover, the solution has compact support under the above conditions, and thus \eqref{1.1} can be regarded as a free boundary problem. For three typical types of $f$ as  monostable case, bistable case and combustion case, the existence of traveling wave solutions, the asymptotic behaviors of the solutions of (\ref{1.1}) as well as the propagation speed of the free boundary with compactly supported initial data were studied in \cite{aronson1982stabilization,aronson1980density,gilding2004travelling,sherratt2010form,du2020logarithmic,garriz2020propagation,de1991travelling,lou2024convergence,du2015nonlinear} and other works.

The recent progress related to problem \eqref{1.1} comes from the paper by Lou and Zhou \cite{lou2024convergence}, they provided a complete characterization for the asymptotic behaviors of the solutions of (\ref{1.1}) with monostable, bistable and combustion types of reactions.
%The authors identified the stationary solutions and established a general convergence theorem: any nonnegative, bounded global solution converges as $t\rightarrow\infty$ to a stationary solution. Based on this general result, they further proved the hair-trigger effect for the monostable porous medium equation (PME), as well as a spreading每transition每vanishing trichotomy governing the asymptotic behavior of solutions to bistable or combustion-type reaction porous medium equations (RPMEs).
Among other things, when $f$ is of combustion type
\begin{equation}
\begin{split}
f(x)=0\  \text{in}\ [0,\theta], \quad f(x)>0\ \text{in}\ (\theta,1), \quad f(x)\leq0\ \text{in}\ [1,\infty), \quad f'(x)>0\ \text{for}\ x\in(\theta,\theta+\sigma),
\end{split}\label{1.4}
\end{equation}
and $u_0=\lambda \psi$ with $\psi$ being chosen from some appropriate set, they proved a trichotomy result: there exists a sharp value $\lambda^{*}>0$ such that
\begin{itemize}
\item if $\lambda>\lambda^{*}$, spreading happens: $\lim_{t\rightarrow\infty}u=1$ locally uniformly as $t\rightarrow\infty$;
\item if $\lambda=\lambda^{*}$, transition happens: $\lim_{t\rightarrow\infty}u=\theta$ locally uniformly as $t\rightarrow\infty$;
\item if $\lambda<\lambda^{*}$, vanishing happens: $\lim_{t\rightarrow\infty}u=0$ locally uniformly as $t\rightarrow\infty$,
\end{itemize}
where $\theta\in(0,1)$ is the socalled ignition temperature of $f$. Moreover, for the critical case $\lambda=\lambda^*$, by denoting the compact support of the solution as spt$[u(\cdot,t)]:=[l(t),r(t)]$ for all $t>0$ (whose existence is well known) they further demonstrated that both $-l(t)$ and $r(t)$ are non-decreasing and satisfy $-l(t),\ r(t)=2y_{0}\sqrt{t}[1+o(1)]\ \text{as} \ t\rightarrow\infty$ (see Theorems \ref{A} and \ref{B} in the next section for details).

The main purpose of this paper is to establish a more refined asymptotic characterization of the free boundary $-l(t),\ r(t)$ on the basis of the aforementioned results for a special type of combustion type source terms $f$. Specifically, we assume that $f$ satisfies the following conditions:
and is further characterized by
\begin{equation}
f(x)=(x-\theta)^{p}\quad \text{in}\ [\theta,\theta+\sigma],\label{1.6}
\end{equation}
for some constants $p>1$ and $\sigma>0$.  The main result of this paper is as follows.
\begin{dl}
Assume that $f\in C^1([0,\infty))\cup C^2([0,\theta)\cup(\theta,\infty))$ is of combustion type satisfying $(\ref{1.4})$ and $(\ref{1.6})$. Let $u(x,t)$ be the solution to $(\ref{1.1})$ with compact support spt$[u(\cdot,t)]:=[l(t),r(t)]$ for which the transition happens. Then the following assertions hold.
\begin{itemize}
\item[\rm(\rmnum{1})] If $p>3$, we have

$$2y_{0}\sqrt{t}+ht^{\frac{1}{2}-\frac{1}{p-1}}\leq -l(t),\ r(t)\leq 2y_{0}\sqrt{t}+Ht^{\frac{1}{2}-\frac{1}{p+1}},\ \ \text{as} \ t\rightarrow\infty,  $$
where $h$, $H$ and $y_{0}$ are positive constants independent of $t$.
\item[\rm(\rmnum{2})] If $1<p\leq3$, we have

$$2y_{0}\sqrt{t}-h\leq -l(t),\ r(t)\leq 2y_{0}\sqrt{t}+Ht^{\frac{1}{2}-\frac{1}{p+1}},\ \ \text{as} \ t\rightarrow\infty,  $$

  where $h$, $H$ and $y_{0}$ are positive constants independent of $t$.
\end{itemize}\label{dl1.1}
\end{dl}

Given that the aforementioned asymptotic characterization is $p$-dependent, our theorem demonstrates the impossibility of establishing a refined asymptotic characterization for general combustion type source terms $f$.

Instead of the Cauchy problem \eqref{1.1}, a type of free boundary problem has been studied by many researchers for the case $m=1$: Du and Lou \cite{du2015spreading} proved the trichotomy result, and Du, Lou and Zhou \cite{du2015nonlinear} established an estimate of the free boundary when the transition case occurs for combustion type $f$. Later, Lei et al. obtained a more precise characterization of the propagation speed for a specific type of $f$, where the propagation speed of the free boundary $h(t)$ essentially depends on $f$.
The works of Lou and Zhou \cite{lou2024convergence}, as well as the main results of this paper, are analogous to those in the aforementioned case of $m=1$. However, since the case $m > 1$ introduces significant structural changes and additional technical challenges, the proof techniques deviate substantially from those in previous works. In \cite{lei2018refined}, the argument hinges on the construction of a self-similar solution $v(x,t)$ to the nonlinear heat equation $v_t = v_{xx} + v^p$. By observing that $v(x,t) + \theta$ satisfies $v_t = v_{xx} + (v-\theta)^p$, the authors established a lower bound for $u(0,t)$ by leveraging the zero-number diminishing property from \cite{angenent1988zero}. Such properties, however, do not extend to the porous medium equation. To overcome this hurdle, we directly construct a pair of sub- and super-solutions for $v_t = (v^m)_{xx} + (v-\theta)^p$, which allows us to recover the same lower bound for $u(0,t)$ and subsequently estimate the lower bound of the free boundaries.

This paper is arranged as follows. In Section \ref{s2}, we present some basic results that are fundamental for this research and will be used in our proofs. In Section \ref{s3}, we give the proofs of our main results: Firstly we give the upper bound estimate of the free boundary $r(t)$ by investigating the level set $\{\theta(t)| u(\theta(t),t)=\theta\}$ in subsection 3.1. Secondly we estimate the lower bound of the free boundary $r(t)$ for $p>3$ in subsection 3.2, the key point is to obtain the lower bound of $u(0,t)$, which is the most challenging part of the paper. In subsection 3.3, we finally  construct suitable lower solutions to obtain a lower bound estimate for the free boundary $r(t)$ for $1<p\leq3$.

\section{Preliminary}
\label{s2}
In this section, we give some basic facts, which will be used frequently in the forthcoming sections. For convenience of discussion, we first present the definition of weak solutions and the restrictions on the initial values.

\begin{dy}
For any $T>0$, denote $\Omega_{T}:=\mathbb{R}\times(0,T)$. A very weak solution $u(x,t)$ of (\ref{1.1}) on $\Omega_{T}$ is a function $u(x,t)\in C(\Omega_{T})\cap L^{\infty}(\Omega_{T})$ with the following properties:
\begin{equation}
\int_{\mathbb{R}}u(x,T)\varphi(x,T)dx=\int_{\mathbb{R}}u_{0}(x)\varphi(x,0)dx+\int\int_{\Omega_{T}}f(u)\varphi dxdt+\int\int_{\Omega_{T}}[u\varphi_{t}+u^{m}\varphi_{xx}]dxdt,\label{1.5}
\end{equation}
for all $\varphi\in C_{c}^{\infty}(\Omega_{T})$, and a weak subsolution (supersolution) $u(x,t)\in C(\Omega_{T})\cap L^{\infty}(\Omega_{T})$ is defined by (\ref{1.5}) with equality replaced by $\leq(\geq)$ and imposing the additional condition $\varphi\ge 0${\rm (}cf. \cite[Chapter 5]{vazquez2007porous}{\rm )}.
\end{dy}
In the rest of this paper, the initial function $u_{0}$ will be chosen from
\begin{equation}
\mathfrak{X}:=
\left\{\psi\in C(\mathbb{R})\left|
\begin{split}
&\text{there exist}\  -\infty<l_{1}<r_{1}\leq l_{2}<r_{2}\leq\cdot\cdot\cdot\leq l_{n}<r_{n}<+\infty\ \text{ such}\\
&\text{that}\ \psi(x)>0\ \text{in}\ (l_{j},r_{j})\  \text{for}\ j=1,...,n,\ \text{and}\ \psi(x)=0\  \text{otherwise}.
\end{split}
\right.
\right\}.\label{ii}
\end{equation}
\indent

Next, we review some of the results from Lou and Zhou's paper that are directly relevant to our research.

\begin{mytheorem}\cite[Theorem 1.4]{lou2024convergence}
Suppose that $f(x)$ is a combustion-type function, characterized by conditions \eqref{ff} and \eqref{1.4}. Let $u(x, t)$ be the solution of $(\ref{1.1})$ with initial datum $u_{0}=\psi\in\mathfrak{X}$. Then one of the following situations occurs:
\begin{itemize}
\item[\rm(\rmnum{1})] \textbf{Vanishing}:
$$\lim_{t\rightarrow\infty}\|u(x,t)\|_{L^{\infty}(\mathbb{R})}=0;$$
\item[\rm(\rmnum{2})] \textbf{Spreading}:
$$\lim_{t\rightarrow\infty}\|u(x,t)-1\|_{C^{2}([-M,M])}=0,\ \forall  M>0;$$
\item[\rm(\rmnum{3})] \textbf{Transition}:
$$\lim_{t\rightarrow \infty} \|u(x,t)-\theta\|_{C^{2}([-M,M])}=0,\ \forall M>0.$$\ \
\end{itemize}
  Moreover, if $u_{0}=\lambda\psi$ for some $\psi\in\mathfrak{X}$, then there exists $\lambda^{*}=\lambda^{*}(\psi)\in(0,\infty)$ such that vanishing happens when $0<\lambda<\lambda^{*}$, spreading happens when $\lambda>\lambda^{*}$, and transition happens when $\lambda=\lambda^{*}$.
\label{A}
\end{mytheorem}

\begin{mytheorem}\cite[Proposition 5.7]{lou2024convergence}
Assume that $f$ is of combustion type and let $u(x,t)$ be the solution of $(\ref{1.1})$ for which the transition happens.  Then the support of $u(x,t)$  is equal to $[l(t),r(t)]$ with the left and
right free boundaries $l(t)$ and $r(t)$. Moreover, the following estimate holds
$$-l(t),r(t)=2y_{0}\sqrt{t}[1+o(1)]\ \text{as} \ t\rightarrow\infty$$
for some $y_{0}\in(0,\theta^{\frac{m-1}{2}})$.\label{B}
\end{mytheorem}

\begin{zy}
To obtain the lower bound of $v_{xx}$ {\rm (}where $v=\frac{m}{m-1}u^{m-1}${\rm )}, paper \cite{lou2024convergence} assumed $f\in C^{2}$. However, when $1\le p<2$, the function $f$ considered in this paper fails to satisfy the smoothness assumption. It should be noted that $f$ is only discontinuous in its second derivative at the point $x=\theta>0$ and $u(x,t)$ is classical in $E_{1}:=\{(x,t)|u(x,t)\geq\theta/2\}$, we can estimate $v_{xx}$ through classical parabolic estimation in $E_{1}$, and use the same argument in \cite{lou2024convergence} to estimate $v_{xx}$ in $E_{2}:=\{(x,t)|u(x,t)<\theta/2\}$. Consequently, the two results established in \cite{lou2024convergence} that were mentioned above still hold under our current assumptions.
\end{zy}

In many cases, it is convenient to consider
$$v(x,t):=\frac{m}{m-1}u^{m-1}(x,t)$$
so-called mathematician's pressure. Using variable $v$, problem (\ref{1.1}) is converted into the following problem
\begin{equation}
\left\{
\begin{split}
&v_{t}=(m-1)vv_{xx}+v_{x}^{2}+g(v),\qquad \,\ \, \,x\in \mathbb{R},\ \,\ t>0,\\
&v(x,0)=v_{0}(x)=\frac{m}{m-1}u^{m-1}_{0},\ \ \, \qquad x\in \mathbb{R},\\
\end{split}\label{2.1**'}
\right.
\end{equation}
where
$$g(v):=m\big(\frac{m-1}{m}v\big)^{\frac{m-2}{m-1}}f\bigg(\big(\frac{m-1}{m}v\big)^{\frac{1}{m-1}}\bigg).$$

As shown in Theorem \ref{A}, when the transition happens, the solution of $(\ref{1.1})$ with initial datum $u_{0}=\psi\in\mathfrak{X}$ satisfies $u(x, t)\rightarrow\theta$   as $t\rightarrow\infty$ locally uniformly in $\mathbb{R}$ (especially, $u(0,t)\rightarrow\theta$). Then $v(0,t)=\frac{m}{m-1}u^{m-1}(0,t)\rightarrow\frac{m}{m-1}\theta^{m-1}$ as $t\rightarrow\infty$. Therefore, the following auxilary problem will play a crucial role in the relevant research:
\begin{equation}
\left\{
\begin{split}
&v_{t}=(m-1)vv_{xx}+v_{x}^{2},\ \ \ \ \quad \ \ 0<x<\rho(t),\ \,\ t>0,\\
&v(0,t)=\Theta:=\frac{m}{m-1}\theta^{m-1},\  \ \ \ t>0,\\
&v(\rho(t),t)=0,\ \,\, \,\, \qquad \qquad \ \quad \ \   t>0,\\
&\rho'(t)=-v_{x}(\rho(t),t),\ \ \,\quad \quad \quad \ \ t>0,
\end{split}\label{2.3**}
\right.
\end{equation}
which is equivalent to the following equation
\begin{equation}
\left\{
\begin{split}
&u_{t}=(u^{m})_{xx},\ \ 0<x<\rho(t),\ \,\ t>0,\\
&u(0,t)=\theta,\ \ u(\rho(t),t)=0,\ \ \ \ t>0.
\end{split}\label{2.3}
\right.
\end{equation}
We consider the self-similar solutions of (\ref{2.3}) with structure $u(x,t)=\xi^{\frac{1}{m}}\left(\frac{x}{2\sqrt{t}}\right)$, which leads to the following problem:
 \begin{equation}
\left\{
\begin{split}
&\xi''(y)=-2y(\xi^{\frac{1}{m}})'(y)=-2\frac{y}{m}\xi^{\frac{1-m}{m}}\xi'(y),\quad\quad y>0,\\
&\xi(0)=\theta^{m},\ \xi'(0)=-2\theta^{\frac{m+1}{2}}.
\end{split}\label{2.2}
\right.
\end{equation}

Next, we recall a relevant result from \cite{lou2024convergence} on the existence of self-similar solutions to (\ref{2.2}), which will be utilized to construct sub-solutions and supersolutions in the subsequent proofs.

\begin{yl}\cite[Lemma 5.6]{lou2024convergence}
There exists $y_0$ satisfying $0< y_{0}<\theta^{\frac{m-1}{2}}$ such that problem $(\ref{2.3})$ has a self-similar solution
 \begin{equation}
e(x,t)=\xi^{\frac{1}{m}}\left(\frac{x}{2\sqrt{t}}\right)\label{2.3*}
\end{equation}
with $\rho(t)=2y_{0}\sqrt{t}$, where $\xi$ satisfies  (\ref{2.2}). \label{yl2.1*}
\end{yl}
\indent

Now we give the relationship between the free boundaries of two transition solutions, which implies that we only need to estimate the solution corresponding to specific initial data.

\begin{yl}
Suppose that $f$ is of combustion type and $u_{0}^{1},u_{0}^{2}\in \mathfrak{X}$ are two functions such that the transition happens for the corresponding solutions $u(x,t;u_{0}^{1})$ and $u(x,t;u_{0}^{2})$. Then there exists a constant $C>0$ such that \label{yl2.1}
$$|r(t;u_{0}^{1})-r(t;u_{0}^{2})|\leq C,\ \ |l(t;u_{0}^{1})-l(t;u_{0}^{2})|\leq C.$$\label{yl2.2}
\end{yl}
\noindent\textbf{Proof.} We omit the proof since it is quite similar to Lemma 4.1 in \cite{du2015nonlinear}. \hfill$\Box$

Thanks to Lemma \ref{yl2.2}, we only need to focus on initial data $\widetilde{u}_{0}\in \mathfrak{X}$ which is symmetrically decreasing with $j=1$. We now proceed to prove that the solution corresponding to this specific initial data possesses certain favorable properties.

\begin{yl}
If $\widetilde{u}_{0}\in \mathfrak{X}$ is symmetrically decreasing with
$j=1$ and the corresponding solution
$u(x,t;\widetilde{u}_{0})$ of $(\ref{1.1})$ undergoes a transition, then $u(x,t;\widetilde{u}_{0})$ is symmetrically decreasing and $u_{x}(x,t;\widetilde{u}_{0})<0$ in $\{(x,t)|x>0, u(x,t;\widetilde{u}_{0})>0\}$. Moreover, $u(0,t;\widetilde{u}_{0})>\theta$ for all $t>0$.\label{yl2.3}
\end{yl}
\noindent\textbf{Proof.} It is routine to check $u(-x,t;\widetilde{u}_{0})$ is also a solution of $(\ref{1.1})$ with the initial datum $\widetilde{u}_{0}\in\mathfrak{X}$. Hence, by the uniqueness of solutions, we immediately obtain the symmetry of
$u(x,t;\widetilde{u}_{0})$.

We now prove that $u_{x}(x^*,t;\widetilde{u}_{0})<0$ holds when $x^*>0$ and $u(x^{*},t;\widetilde{u}_{0})>0$. Since the transition happens, it follows that $r(t)\rightarrow\infty$ as $t\rightarrow\infty$, by {\bf Positivity persistence} (see \cite[Section 2.1]{lou2024convergence}), there exists $t_0\in [0,\infty)$ such that $(x^*,t)\in\{(x,t)|x>0, u(x,t;\widetilde{u}_{0})>0\}$ iff $t\ge t_0.$

Let
$$v(x,t):=\frac{m}{m-1}u^{m-1}(x,t;\widetilde{u}_{0})$$
and
$$\eta(x,t):=v(x,t)-v(2x^{*}-x,t), \ \ \ \ (x,t)\in E_{1}:=\{(x,t)|-\infty\leq x\leq x^{*},\ t\geq t_{0}\}.$$
Then there exist two funcitons $c_{1}=c_{1}(v_{x})$ and $c_{2}=c_{2}(m,f,v,v_{xx})$ such that $\eta$ satisfies
\begin{equation}
\left\{
\begin{split}
&\eta_{t}=(m-1)v\eta_{xx}+c_{1}\eta_{x}+c_{2}\eta,\qquad \, (x,t)\in E_{1},\\\
&\eta(x^{*},t)=0,\ \ \ \ \qquad\qquad\qquad\qquad\quad t\geq t_{0},\\
&\eta(x,t_{0})\geq0,\ \ \ \qquad\qquad\qquad\qquad\quad \ x\leq x^{*}.\\
\end{split}
\right.
\end{equation}
The maximum principle implies that
$$\eta(x,t)\geq0\ \text{in}\ E_{1},\ \eta_{x}(x^{*},t)\leq0\  \text{for}\  t\geq t_{0}.$$
Furthermore, since $v(x^{*},t)>0$ for $t>t_{0}$, so both $v$ and $\eta$ are classical
near the line $\{x=x^{*}, t>t_{0}\}$, we can employ the Hopf lemma to derive
 $$\eta_{x}(x^{*},t)=2v_{x}(x^{*},t)<0\ \ \text{for}\ t>t_{0}, $$
which implies $u_{x}(x^{*},t;\widetilde{u}_{0})<0$, for $t>t_0$ immediately.

At last, we show that $u(0,t;\widetilde{u}_{0})>\theta$ for all $t>0$. Assume by contradiction that there exists $t_{1}>0$ such that $u(0,t_{1};\widetilde{u}_{0})\le\theta$. The property $u_{x}(x,t;\widetilde{u}_{0})<0$ for $x>0$ implies that
$$u(x,t_{1};\widetilde{u}_{0})<\theta,\ \ \text{for} \ \ x\in[-r(t_{1}),r(t_{1})]\backslash\{0\}.$$
By the strong maximum principle we see that
$$u(x,t_{1};\widetilde{u}_{0})<\theta,\ \ \text{for} \ \ t>t_{1}\ \ \text{and}\ x\in[-r(t),r(t)],$$
and hence $u(x,t;\widetilde{u}_{0})\rightarrow0$ as $t\rightarrow\infty$ which is contradicts to the assumption that the transition happens. The proof is complete.\hfill$\Box$

\section{Proof of the main theorem}
\label{s3}
\qquad Throughout this section, we assume that $f$ is of combustion type and $(\ref{1.6})$ holds for $p>1$ and $\sigma>0$, and that $u(x,t)$ is the solution of $(\ref{1.1})$ with the initial datum $\widetilde{u}_{0}\in\mathfrak{X}$.
%This section is devoted to establishing the upper bound and lower bound of $r(t)$.

\subsection{Upper bound estimate of $r(t)$}

 \qquad By Lemma \ref{yl2.3}, we see that for each $t\geq0$, there exists a unique $\theta(t)\in(0,r(t))$ such that
 $$u(\theta(t),t)=\theta.$$
We intend to establish the estimate of $\theta(t)$, combining with a special self-similar solution of $(\ref{2.3})$, then by constructing the suitable sub-solution and  super-solution of $(\ref{1.1})$, we can obtain a lower and upper bound of $r(t)$ by comparison principle.\\
\indent
Given a constant $b\in(0,1-\theta)$, consider the initial value problem:
\begin{equation}
(Q^{m})''+f(Q)=0,\ \ Q(0)=\theta+b,\ \  Q'(0)=0\ \text{in}\ (0,\infty),\label{3.1}
\end{equation}
through a phase plane analysis, we know that $(\ref{3.1})$ has a unique solution $Q_{b}$ and numbers $l(b)$, $L(b)\in (0,\infty)$
 with $l(b)< L(b)$ such that
\begin{align*}
&Q_{b}(L(b))=0,\ Q_{b}(l(b))=\theta,\\
&Q'_{b}(x)<0\ \ \text{for}\ \ x\in (0,L(b)].
\end{align*}
For $l(b)$ and $L(b)$, we have the follwing lemma.
\begin{yl}
If $p>1$, $b\in(0,\sigma)$, then $l(b)$, $L(b)$ are strictly decreasing in $b$ and $l(b)$, $L(b)\rightarrow\infty $ if and only if  $b\rightarrow0$.\label{yl3.1}
\end{yl}
\noindent\textbf{Proof.} First, we rewrite the equation  $(\ref{3.1})$ into the following form
\begin{align*}
&(Q_{b}^{m})''(Q_{b}^{m})' =-f(Q)(Q_{b}^{m})',
\end{align*}
integrating this identity over $[0,x]$, we have
\begin{equation}
(Q_{b}^{m})'(x)=-\bigg(2m\int_{Q_{b}(x)}^{b+\theta}r^{m-1}f(r)dr\bigg)^{\frac{1}{2}}.\label{3.2}
\end{equation}
Using the implicit function theorem, we have
\begin{equation}
l(b)=\int_{\theta}^{\theta+b}\frac{mQ^{m-1}dQ}{\left(2m\int_{Q}^{b+\theta}r^{m-1}f(r)dr\right)^{\frac{1}{2}}}\label{3.3}
\end{equation}
and
\begin{align*}
\int_{\theta}^{\theta+b}\frac{mQ^{m-1}dQ}{\left(2m\int_{Q}^{b+\theta}r^{m-1}f(r)dr\right)^{\frac{1}{2}}}&\leq\frac{\sqrt{m}(\theta+b)^{m-1}}{\sqrt{2}\theta^{\frac{m-1}{2}}}
\int_{\theta}^{\theta+b}\frac{dQ}{\sqrt{\int_{Q}^{b+\theta}(r-\theta)^{p}dr}}\\
&=\frac{\sqrt{m}(\theta+b)^{m-1}}{\sqrt{2}\theta^{\frac{m-1}{2}}}\int_{\theta}^{\theta+b}\frac{dQ}{\sqrt{\frac{1}{p+1}(b^{p+1}-(Q-\theta)^{p+1})}}\\
&=\frac{\sqrt{m}(\theta+b)^{m-1}}{\sqrt{2}\theta^{\frac{m-1}{2}}}\int_{0}^{b}\frac{ds}{\sqrt{\frac{1}{p+1}(b^{p+1}-s^{p+1})}}\\
&=\frac{\sqrt{m}(\theta+b)^{m-1}}{\sqrt{\frac{2}{p+1}}\theta^{\frac{m-1}{2}}b^{\frac{p+1}{2}}}b\int_{0}^{1}\frac{dx}{\sqrt{1-x^{p+1}}}\\
&= C_{p}(b)b^{-\frac{p-1}{2}},
\end{align*}
and similarly we also have
\begin{align*}
\int_{\theta}^{\theta+b}\frac{mQ^{m-1}dQ}{\left(2m\int_{Q}^{b+\theta}r^{m-1}f(r)dr\right)^{\frac{1}{2}}}\geq C_{p}'(b)b^{-\frac{p-1}{2}},
\end{align*}
where $C_{p}(b)$ and $C_{p}'(b)$ are bounded in $(0,\sigma)$ and $\lim_{b\rightarrow0}C_{p}(b)=\lim_{b\rightarrow0}C_{p}'(b)=C>0$. Thus $l(b)\rightarrow\infty $ if and only if  $b\rightarrow0$ and
\begin{align}
l(b)=C_{1}(b)b^{-\frac{p-1}{2}}, \label{3.4*}
\end{align}
where $C'_{p}(b)\leq C_{1}(b)\leq C_{p}(b)$.

Moreover,
\begin{equation*}
\begin{split}
l(b)&=\int_{\theta}^{\theta+b}\frac{mQ^{m-1}dQ}{\left(2m\int_{Q}^{b+\theta}r^{m-1}f(r)dr\right)^{\frac{1}{2}}}\\
&=\int_{\theta}^{\theta+b}\frac{dQ^{m}}{\left(2\int_{Q}^{\theta+b}(r-\theta)^{p}dr^{m}\right)^{\frac{1}{2}}}\\
&=\int_{\theta^{m}}^{\theta^{m}+a}\frac{ds}{\left(2\int_{s}^{\theta^{m}+a}(t^{1/m}-\theta)^{p}dt\right)^{\frac{1}{2}}}\\
&=\int_{0}^{a}\frac{dr}{\left(2\int_{r}^{a}[(i+\theta^{m})^{1/m}-\theta]^{p}di\right)^{\frac{1}{2}}}.\\
\end{split}
\end{equation*}
where $(\theta+b)^{m}=\theta^{m}+a$. Let $\epsilon>0$  is sufficiently small, we have
\begin{equation*}
\begin{split}
\int_{0}^{a+\epsilon}\frac{dr}{\left(2\int_{r}^{a+\epsilon}[(i+\theta^{m})^{1/m}-\theta]^{p}di\right)^{\frac{1}{2}}}
&=\int_{0}^{a+\epsilon}\frac{d\frac{a}{a+\epsilon}r}{\left(2\int_{r}^{a+\epsilon}\frac{a}{a+\epsilon}[(i+\theta^{m})^{1/m}-\theta]^{p}d\frac{a}{a+\epsilon}i\right)^{\frac{1}{2}}}\\
&=\int_{0}^{a}\frac{dt}{\left(2\int_{t}^{a}\left[\left(\frac{a}{a+\epsilon}\right)^{\frac{1}{p}}(s(1+\epsilon/a)+\theta^{m})^{1/m}-\theta\right]^{p}ds\right)^{\frac{1}{2}}}.
\end{split}
\end{equation*}
For simplicity, we set
\begin{equation*}
\begin{split}
g(s):=(s+\theta^{m})^{1/m}-\theta,\quad
h(s):=\left(\frac{a}{a+\epsilon}\right)^{\frac{1}{p}}(s(1+\epsilon/a)+\theta^{m})^{1/m}-\theta).
\end{split}
\end{equation*}
Then $g(0)=h(0)=0$, and
\begin{equation*}
\begin{split}
g'(s)-h'(s)&=\frac{1}{m}(s+\theta^{m})^{\frac{1}{m}-1}-\frac{1}{m}(1+\epsilon/a)^{1-\frac{1}{p}}(s(1+\epsilon/a)+\theta^{m})^{\frac{1}{m}-1}.
\end{split}
\end{equation*}
\begin{equation*}
\begin{split}
\left(\frac{s(1+\epsilon/a)+\theta^{m}}{s+\theta^{m}}\right)^{1-\frac{1}{m}}-\left(1+\frac{\epsilon}{a}\right)^{1-\frac{1}{p}}
&=\left(1+\frac{s}{s+\theta^{m}}\cdot\frac{\epsilon}{a}\right)^{1-\frac{1}{m}}-\left(1+\frac{\epsilon}{a}\right)^{1-\frac{1}{p}}\\
&\leq\left(1+\frac{\epsilon}{a+\theta^{m}}\right)^{1-\frac{1}{m}}-\left(1+\frac{\epsilon}{a}\right)^{1-\frac{1}{p}}\\
&\leq\left(1-\frac{1}{m}\right)\frac{\epsilon}{\theta^{m}}-\left(1-\frac{1}{p}\right)\frac{\epsilon}{a}+o(\epsilon),
\end{split}
\end{equation*}
if $\sigma$ satisfies $(\theta+\sigma)^{m}-\theta^{m}<\theta^{m}m(p-1)/p(m-1)$ and $\sigma<1-\theta$, then $a<\theta^{m}m(p-1)/p(m-1)$, we can derive that $g'(s)<h'(s)$, which implies for any $b\in(0,\sigma)$, there exists a constant $\delta_{b}$ such that for any $b'\in(b,b+\delta_{b})$, $l(b)>l(b')$. \\
\indent
Define $\kappa:=\text{sup}\{k\,|\,l(b)>l(b')\ \text{for}\ b'\in(b,b+k)\ \text{and}\ k<\sigma-b\}$, we claim that $\kappa=\sigma-b$. Suppose to the contrary, $\kappa<\sigma-b$, we have $l(b)=l(b+\kappa)$, then there exists $\lambda\in(b,b+\kappa)$ such that $\lambda$ is the local maximum point of $l(b)$. However, the above proof indicates that any $y\in(0,\sigma)$ can not be a local maximum point, which is a contradiction. Thus we have $\kappa=\sigma-b$, which means for any $b'\in(b,\sigma)$, $l(b)>l(b')$. From the arbitrariness of $b$, we obtain that $l(b)$ is strictly decreasing in $b\in(0,\sigma)$.\\
\indent
Let $x=l(b)$ in $(\ref{3.2})$, we obtain
\begin{equation}
\begin{split}
(Q^{m})'(l(b))&=-\bigg(2m\int_{\theta}^{b+\theta}r^{m-1}f(r)dr\bigg)^{\frac{1}{2}}\\
&=-\bigg(2m\xi^{m-1}\int_{\theta}^{b+\theta}f(r)dr\bigg)^{\frac{1}{2}}\\
&=-\bigg(2m\xi^{m-1}\frac{1}{p+1}b^{p+1}\bigg)^{\frac{1}{2}},
\end{split}\label{3.4}
\end{equation}
where $\xi\in(\theta,\theta+b)$.

As $Q_{b}(x)$ is a linear function over $[l(b), L(b)]$, we have
$$-Q'(l(b))=\frac{\theta}{L(b)-l(b)}=-\frac{(Q^{m})'(l(b))}{m\theta^{m-1}},$$
which is reorganized as follows
\begin{equation}
L(b)=l(b)+\frac{m\theta^{m}}{\left(2m\xi^{m-1}\frac{1}{p+1}\right)^{\frac{1}{2}}}b^{-\frac{p+1}{2}}=l(b)+C_{2}(b)b^{-\frac{p+1}{2}}\label{3.5}
\end{equation}
where $C_{2}(b):=m\theta^{m}/\left(2m\xi^{m-1}\frac{1}{p+1}\right)^{\frac{1}{2}}$. From (\ref{3.4}), we see that $(Q^{m})'(l(b))$ is decreasing in $b$ and $\theta/(Q^{m})'(l(b))$ is increasing in $b$, combining with $l(b)$ strictly decreasing in $b$, we have $L(b)$ is strictly decreasing in $b\in(0,\sigma)$ and $L(b)\rightarrow\infty $ if and only if $b\rightarrow0$.\ The proof is complete. \hfill$\Box$

For any fixed $t\geq1$, the sign-changing pattern of the function $u(x,t)-Q_{b}(x)$ is crucial for obtaining a refined estimate for $\theta(t)$.\\
\indent
We write $w_{b}(x,t):=u(x,t)-Q_{b}(x)$,
similarly to Lemma 4.5 of \cite{du2015nonlinear}, there exists $\delta>0$ such that for each $b\in(0,\delta)$,
$$w_{b}(0,1)>0>w_{b}(r(1),1),$$
and has a unique zero in $[0, r(1)]$, and the zero is nondegenerate.\\
\indent
Since $t\mapsto w_{b}(x,t),\ t\mapsto (w_{b})_{x}(x,t)$ and $t\mapsto r(t)$ are all uniformly continuous in $x$, for each $b\in(0,\delta)$ there exists $\epsilon_{0}>0$ small such that for each fixed $t\in[1-\epsilon_{0},1+\epsilon_{0}]$, $w_{b}(x,t)$ satisfies
$$w_{b}(0,t)>0>w_{b}(r(t),t).$$
and has a unique zero in $[0,r(t)]$, which is nondegenedete.\\
\indent
Define
\begin{align*}
&T^{b}_{1}:=\text{sup}\{s:w_{b}(0,t)>0\ \ \text{for}\  t\in[1-\epsilon_{0},s)\},\\
&T^{b}_{2}:=\text{sup}\{s:r(t)<L(b)\ \ \text{for}\  t\in[1-\epsilon_{0},s)\}.
\end{align*}
Clearly $T^{b}_{1},\ T^{b}_{2}\geq1+\epsilon_{0}$. Due to $r(t)\rightarrow\infty$ and $w_{b}(0,t)\rightarrow-b<0$ as $t\rightarrow\infty$, both $T^{b}_{1}$ and $T^{b}_{2}$ are finite. To simplify notation, we will write $w(x,t)$, $T_{1}$ and $T_{2}$ instead of $w_{b}(x,t)$, $T^{b}_{1}$ and $T^{b}_{2}$ when the dependence of $b\in(0,\delta)$ is not stressed.\\
\indent
The following lemmas provide all possible case for the intersection points between $u(x,t)$ and the solution of $(\ref{3.1})$ when transition case happens.
\begin{yl}Suppose $T_{1}<T_{2}$. Then
\begin{itemize}
\item[\rm(\rmnum{1})]  for $t\in(1,T_{1})$, $w(x,t)$ has a unique nondegenerate zero $x(t)$ in $(0,r(t))$ with
sign-changing pattern $[+0-]$ over $[0,r(t)]$, meaning
$$w(x,t)>0\  in\  [0,x(t)),\  w(x(t),t)=0,\  w(x,t)<0\  in\  (x(t),r(t)];$$
\item[\rm(\rmnum{2})]  $w(x,T_{1})$ has sign-changing pattern $[0-]$ over $[0,r(T_{1})]$, meaning
$$w(0,T_{1})=0\  and\  w(x,T_{1})<0\  in\  (0,r(T_{1})];$$
\item[\rm(\rmnum{3})] for $t\in(T_{1},T_{2})$, $w(x,t)$ has sign-changing pattern $[-]$ over $[0,r(t)]$, meaning
$$w(x,t)<0\  in\  [0,r(t)];$$
\item[\rm(\rmnum{4})]  $w(x,T_{2})$ has sign-changing pattern $[-0]$ over $[0,L(b)]$, meaning
$$w(x,T_{2})<0\   in\  [0,r(T_{2}))\  and\  w(r(T_{2}),T_{2})=0;$$
\item[\rm(\rmnum{5})]  for $t>T_{2}$, $w(x,t)$ has a unique nondegenerate zero $y(t)$ in $(0,L(b))$, with
sign-changing pattern $[-0+]$ over $[0,L(b)]$, meaning
$$w(x,t)<0\  in\  [0,y(t)),\  w(y(t),t)=0,\  w(x,t)>0\  in\  (y(t),L(b)];$$
\item[\rm(\rmnum{6})]  $\lim_{t\nearrow T_{1}}x(t)=0$, $\lim_{t\searrow T_{2}}y(t)=L(b)$.
\end{itemize}\label{yl3.2*}
\end{yl}

\begin{yl}Suppose $T_{1}>T_{2}$. Then
\begin{itemize}
\item[\rm(\rmnum{1})]  for $t\in(1,T_{2})$, $w(x,t)$ has a unique nondegenerate zero $x(t)$ in $(0,r(t))$ with
sign-changing pattern $[+0-]$ over $[0,r(t)]$;
\item[\rm(\rmnum{2})]  $w(x,T_{2})$ has a unique nondegenerate zero $x(T_{2})$ in $(0,L(b))$ plus a second zero
at $x=L(b)$, and it has sign-changing pattern $[+0-0]$ over $[0,L(b)]$;
\item[\rm(\rmnum{3})] for $t\in(T_{2},T_{1})$, $w(x,t)$ has exactly two nondegenerate zeros $x(t)<y(t)$ in $[0,L(b)]$ with sign-changing pattern $[+0-0+]$ over $[0,L(b)]$;
\item[\rm(\rmnum{4})]  $w(x,T_{1})$ has a unique nondegenerate zero $y(T_{1})$ in $(0,L(b))$ plus a second zero
at $x=0$, and it has sign-changing pattern $[0-0+]$ over $[0,L(b)]$;
\item[\rm(\rmnum{5})]  for $t>T_{1}$, $w(x,t)$ has a unique nondegenerate zero $y(t)$ in $(0,L(b))$, with
sign-changing pattern $[-0+]$ over $[0,L(b)]$;
\item[\rm(\rmnum{6})]   $x(t)$ is a $C^{1}$ function for $t\in[1,T_{1}]$ with $\lim_{t\nearrow T_{1}}x(t)=0$, and $y(t)$ is a $C^{1}$
function for $t>T_{2}$ with $\lim_{t\searrow T_{2}}y(t)=L(b)$.
\end{itemize}\label{yl3.3*}
\end{yl}

\begin{yl}Suppose $T_{1}=T_{2}$. Then
\begin{itemize}
\item[\rm(\rmnum{1})]  for $t\in(1,T_{2})$, $w(x,t)$ has a unique nondegenerate zero $x(t)$ in $(0,r(t))$ with
sign-changing pattern $[+0-]$ over $[0,r(t)]$;
\item[\rm(\rmnum{2})]   $w(x,T_{1})$ has sign-changing pattern $[0-0]$ over $[0,L(b)]$;
\item[\rm(\rmnum{3})] for $t>T_{1}=T_{2}$, $w(x,t)$ has a unique nondegenerate zero $y(t)$ in $(0,L(b))$
with sign-changing pattern $[-0+]$ over $[0,L(b)]$;
\item[\rm(\rmnum{4})]  $\lim_{t\nearrow T_{1}}x(t)=0$, $\lim_{t\searrow T_{1}}y(t)=L(b)$.
\end{itemize}\label{yl3.4*}
\end{yl}
\indent
 By the zero number argument and  the Darcy law for porous media equations with reactions, the above lemmas can be proved by the same arguments as Lemmas 4.6-4.8 in \cite{du2015nonlinear} (see \cite[Proposition 5.7]{lou2024convergence}).\\
\indent
Now we are ready to present the estimate of $\theta(t)$ in the following lemma.
\begin{yl}
Let $u(x,t)$ be the solution of $(\ref{1.1})$ with the initial datum $\widetilde{u}_{0}\in\mathfrak{X}$. Then there exists $G>0$ such that
\begin{equation}\label{est of theta}\theta(t)\leq Gt^{\frac{1}{2}-\frac{1}{p+1}}  \mbox{~for~all~large~} t.
\end{equation}
\label{yl3.2}
\end{yl}\vspace{-2em}
\noindent\textbf{Proof.} Let $\sigma_{1}\in(0,\sigma)$ be arbitrarily given. By Lemma \ref{yl3.1}, for any $y\in(L(\sigma_{1}),\infty)$, there
exists a unique $b\in(0,\sigma_{1})$ such that $L(b)=y$. Since $r(t)\rightarrow\infty$ as $t\rightarrow\infty$ and $r(t_{2})>r(t_{1})$ for $t_{2}>t_{1}$ (see \cite[Theorem 2.7]{lou2024convergence}), there exists
$T_{0}>0$ such that
$$r(t)>L(\sigma_{1})\ \ \text{for}\ \ t\geq T_{0}.$$
Hence for any $t\geq T_{0}$, there exists a unique $b(t)\in(0,\sigma_{1})$ such that
$$r(t)=L(b(t))\ \ \text{for}\ \ t\geq T_{0}.$$
In order to estimate $\theta(t)$, we first claim that $\theta(t)\leq l(b(t))$ for all large $t$.

%By studying the intersection point of $u(x,t)$ and $Q_{b(t)}(x)$, we have the following claim.
Since $\lim_{t\rightarrow\infty}u(x,t)=\theta$ locally uniformly in $\mathbb{R}^{1}$, there exists $T_{0}'>T_{0}$ such that
\begin{align}
u(0,t)<\theta+\sigma_{1}\ \ \text{for}\ \ t\geq T_{0}'. \label{3.6}
\end{align}
It is sufficient to prove that
$$\theta(t)\leq l(b(t))\ \ \text{for}\ \ t\geq T_{0}'.$$
\indent
Suppose claim is not true, there exists $t_{0}>T_{0}'$ such that $\theta(t_{0})>l(b(t_{0})):=l(b_{0})$.
Since $r(t_{0})=L(t_{0})$, By Lemmas \ref{yl3.2*}-\ref{yl3.4*}, we have
\begin{align*}
u(x,t_{0})-Q_{b_{0}}(x)<0\ \ \text{for}\ \ x<L(b_{0}) \text{ but close to}\ L(b_{0}).
\end{align*}
Note that $\theta(t_{0})> l(b_{0})$ and $u_{x}(x,t)<0$ in $(0,r(t))$, we have
\begin{align}
u(t_{0},l(b_{0}))-Q_{b_{0}}(l(b_{0}))>0. \label{3.8*}
\end{align}
Hence $u(t_{0},x)-Q_{b_{0}}(x)$ has a zero in $(l(b_{0}),L(b_{0}))$, which can only happen in (ii) of Lemma \ref{yl3.3*}. Then $u(t_{0},x)-Q_{b_{0}}(x)$ has the sign-changing pattern of $[+0-0]$, combining with $(\ref{3.8*})$, we have
\begin{align*}
&u(t_{0},x)-Q_{b_{0}}(x)>0,\ \ \text{for}\ 0\leq x<l(b_{0}),\\
&u(t_{0},0)-b_{0}-\theta>0.
\end{align*}
\indent Define $
A:=\{b>b_{0}:u(x,t_{0})\geq Q_{b}(x)\geq \theta\  \text{in} \ [0,l(b)]\},\
b_{*}:=\text{sup}A.$
%First we show that set $A$ is nonempty.
 Note that
\begin{align}
u(x,t_{0})>Q_{b_{0}}(x)>\theta,\ \ \text{for}\ 0\leq x\leq l(b_{0}),\label{3.7}
\end{align}
and since $l(b)$ is strictly decreasing in $b$ (due to Lemma \ref{yl3.1}) and $Q_{b}$ is continuous with respect to $b$, we obtain
$$u(x,t_{0})>Q_{b_{2}}(x)>\theta,\ \ \text{for}\ 0\leq x\leq l(b_{2})$$
for $b_{2}>b_{0}$ but close to $b_{0}$. This implies that $A$ is nonempty.
Since for all $b\in A$, $u(0,t_{0})\geq b+\theta$, combining with (\ref{3.6}), we have
$$ b\leq u(0,t_{0})-\theta\leq\sigma_{1}.$$
Then $b_{*}\leq u(0,t_{0})-\theta\leq\sigma_{1}$. Therefore $b_{*}$ is well-defined.

 Next we prove that
$b_{*}=u(0,t_{0})-\theta.$
Suppose that $b_{*}<u(0,t_{0})-\theta$. Then $Q_{b_{*}}(0)<u(0,t_{0})$, that is, $w_{b_{*}}(0,t_{0})>0$. Due to $b_{*}>b_{0}$ and $L(b)$ is strictly decreasing in $b$, we have
$$L(b_{*})<L(b_{0})=L(b(t_{0}))=r(t_{0}).$$
From Lemmas \ref{yl3.2*}-\ref{yl3.4*}, we get
$$T_{2}^{b_{*}}<t_{0}<T_{1}^{b_{*}},$$
and $u(x,t_{0})-Q_{b_{0}}(x)$ has the sign-changing pattern of $[+0-0+]$. \\
\indent
Since $l(b)$ is strictly decreasing in $b$, then for any $b\in(b_{0},b_{*})$, it follows that
$$u(x,t_{0})\geq Q_{b}(x),\ \ \text{for all}\ x\in[0,l(b_{*})]\subset[0,l(b)].$$
Letting $b\nearrow b^{*}$ gives
$$u(x,t_{0})\geq Q_{b_{*}}(x),\ \ \text{for}\ 0\leq x\leq l(b_{*}).$$
 Note that
$$Q_{b_{*}}(l(b_{*}))=Q_{b_{0}}(l(b_{0}))=\theta\ \ \text{and}\ \ u_{x}(x,t)<0.$$
Combining with $(\ref{3.8*})$, we obtain
$$u(l(b_{*}),t)-Q_{b_{*}}(l(b_{*}))>u(l(b_{0}),t)-Q_{b_{0}}(l(b_{0}))>0.$$
Since $u(x,t_{0})-Q_{b_{*}}(x)$ has the sign-changing pattern $[+0-0+]$, we have
$$u(x,t_{0})>Q_{b_{*}}(x),\ \ \text{for}\ 0\leq x\leq l(b_{*}).$$
By continuity of $Q_{b}$ with respect to $b$, we can find $\epsilon>0$ such that
$$u(x,t_{0})>Q_{b_{*}+\epsilon}(x),\ \ \text{for}\ 0\leq x\leq l(b_{*}+\epsilon),$$
which contradicts the definition of $b_{*}$. Consequently, we have $b_{*}=u(0,t_{0})-\theta$, that is, $u(0,t_{0})=Q_{b_{*}}(0)$ and
$$u(x,t_{0})\geq Q_{b_{*}}(x),\ \ \text{for}\ 0\leq x\leq l(b_{*}).$$
However, by Lemmas \ref{yl3.2*}-\ref{yl3.4*}, we have that if $u(x,t_{0})-Q_{b_{*}}(x)$ has a tangency at $x=0$, then there exists sufficiently small $\delta>0$ such that $u(x,t_{0})<Q_{b_{*}}(x)$ for  $x\in(0,\delta)$, which derives a contradiction, which complete the proof of the claim.

Now, based on the above claim, we can further establish \eqref{est of theta}.
We recall Theorem \ref{B}, (\ref{3.5}) and Lemma \ref{yl3.1}, i.e.,
\begin{align}
&r(t)=2y_{0}\sqrt{t}[1+o(1)]\ \ \text{as}\ \ t\rightarrow\infty,\label{3.8}\\
&r(t)=L(b(t))=C_{1}(b(t))b(t)^{-\frac{p-1}{2}}+C_{2}(b(t))b(t)^{-\frac{p+1}{2}},\label{3.9}
\end{align}
$$b\rightarrow0\ \ \text{ if and only if}\ \ L(b)\rightarrow\infty.$$
Therefore, due to $r(t)\rightarrow\infty$ as $t\rightarrow\infty$, by Lemma \ref{yl3.1}, we obtain
$$b(t)\rightarrow0\ \ \text{as}\ \ t\rightarrow\infty.$$
Since $\lim_{t\rightarrow\infty}C_{1}(b(t))=I>0$ and $\lim_{t\rightarrow\infty}C_{2}(b(t))=I'>0$, from (\ref{3.8}) and (\ref{3.9}), we have
$$2y_{0}\sqrt{t}[1+o(1)]=Cb(t)^{-\frac{p+1}{2}}[1+o(1)]\ \ \text{as}\ \ t\rightarrow\infty,$$
consequently, we obtain
$$b(t)=O(t^{-\frac{1}{p+1}})\ \ \text{as}\ \ t\rightarrow\infty.$$
This and (\ref{3.4*}) yield
$$l(b(t))=O(t^{\frac{1}{2}-\frac{1}{p+1}})\ \ \text{as}\ \ t\rightarrow\infty.$$
Therefore there exists $G>0$ such that
$$\theta(t)\leq l(b(t))\leq Gt^{\frac{1}{2}-\frac{1}{p+1}}\ \ \text{for all large}\ t,$$
This proves our lemma.\hfill$\Box$

With the above preparations, we can proceed to estimate an upper bound for $r(t)$.
\begin{yl}
If $p>1$ and let $u(x,t)$ be the solution of (\ref{1.1}) with the initial datum $\widetilde{u}_{0}\in\mathfrak{X}$ , we have
$$r(t)\leq 2y_{0}\sqrt{t}+Mt^{\frac{1}{2}-\frac{1}{p+1}}\ \ \text{as} \ \ t\rightarrow\infty.$$
where $y_{0}>0$ is the zero of $\xi(y)$.\label{yl3.3}
\end{yl}
\noindent\textbf{Proof.}  In view of Lemma \ref{yl3.2}, we can find some $K>0$ and $T_{0}>1$ such that
$$\theta(t)\leq Kt^{\frac{1}{2}-\frac{1}{p+1}}\ \ \text{for}\ \  t\geq T_{0}.$$
Fix $T_{0}$ and choose $M>0$ so that
$$M\geq \textup{max}\{K,\ T_{0}^{\frac{1}{p+1}-\frac{1}{2}}r(T_{0})\}.$$
Define
\begin{equation*}
\begin{split}
&h(t):=Mt^{\frac{1}{2}-\frac{1}{p+1}},\\
&\overline{u}(x,t):=e(x-h(t),t),\quad\ h(t)\leq x\leq h(t)+\rho(t),\ \ \ t\geq T_{0},
\end{split}
\end{equation*}
where $e(x,t)$ is designated in (\ref{2.3*}) and $\rho(t)=2y_{0}\sqrt{t}$.

We see that $0=u(x,T_{0})\leq \overline{u}(x,T_{0})$ for $h(T_{0})\leq x\leq h(T_{0})+\rho(T_{0})$. Set $E_{1}:=\{(x,t)|h(t)\leq x\leq h(t)+\rho(t),\ T_{0}\leq t\leq \infty\}$. According to Lemma \ref{yl2.1*}, we have $e_{x}(x-h(t),t)=\frac{1}{2m\sqrt{t}}\xi^{\frac{1}{m}-1}\xi'((x-h(t))/2\sqrt{t})<0$,  and noting that $\xi(x)$ is a solution of (\ref{2.2}) together with $h'(t)>0$, a direct calculation yields
\begin{equation*}
\begin{split}
\overline{u}_{t}-(\overline{u}^{m})_{xx}&=e_{t}(x-h(t),t)-h'(t)e_{x}(x-h(t),t)-(e^{m})_{xx}(x-h(t),t)\\
&=-h'(t)e_{x}(x-h(t),t)\\
&\geq0=f(\overline{u})
\end{split}
\end{equation*}
for $(x,t)\in E_{1}$. Consequently, $\overline{u}(x,t)$ is a supersolution of (\ref{1.1}) in $E_{1}$.
By comparison principle (see, for example, \cite{aronson1982stabilization}), we have
$$r(t)\leq h(t)+\rho(t)\leq2y_{0}\sqrt{t}+Mt^{\frac{1}{2}-\frac{1}{p+1}}\ \ \text{for} \ \ t\geq T_{0}.$$
The proof is complete. \hfill$\Box$

\subsection{Lower bound estimate of $r(t)$ for $p>3$}

$\quad$ In this subsection, we first establish a lower bound for $u(0, t)$, and subsequently utilize the self-similar solution of (\ref{2.3}) to construct a suitable sub-solution, thereby deriving the lower bound for $r(t)$.

\begin{yl}
Assume that $p>3$ and let $u(x,t)$ be the solution of $(\ref{1.1})$ with the initial datum $\widetilde{u}_{0}\in\mathfrak{X}$.
Then there exist constants $C>0$ and $T>0$ such that
$$u(0,t)\geq \theta+C(t+T)^{-\frac{1}{p-1}}\ \ \text{for all} \ \ t\geq T.$$\label{yl3.4}
\end{yl}\vspace{-2em}
\noindent\textbf{Proof.}
We first consider the self-similar solution of the following nonlinear heat equation\\
\begin{equation}
n_{t}=m\theta^{m-1}n_{xx}+|n|^{p-1}n,\label{3.11}
\end{equation}
and its associated problem
\begin{equation}
\left\{
\begin{split}
&\frac{\varphi}{p-1}+\frac{y}{2}\varphi'+m\theta^{m-1}\varphi''+|\varphi|^{p-1}\varphi=0, \ \ y>0,\\
&\varphi(0)=\gamma,\ \varphi'(0)=0.
\end{split}
\right.\label{3.12*}
\end{equation}
Following the ideas of Haraux and Weissler \cite{haraux1982non}, one can prove that for
$$p>3\ \ \text{and}\ \ 0<\gamma<\bigg(\frac{p-3}{2p-2}\bigg)^{\frac{1}{p-1}},$$
the above problem has a unique solution $\varphi\in C^{3}(\mathbb{R})$, which satisfies $\varphi(y)=\varphi(-y)$, $\varphi(y)>0$ for all $y\in[0,\infty)$ and
\begin{equation}
\lim_{y\rightarrow\infty}y^{2/(p-1)}\varphi(y)=A>0.\label{3.13*}
\end{equation}
We rewrite the first equation in $(\ref{3.12*})$ in the form
$$\Big(m\theta^{m-1}\varphi'(y)e^{y^{2}/(4m\theta^{m-1})}\Big)'=-\Big(\frac{\varphi(y)}{p-1}+\varphi^{p}(y)\Big)e^{y^{2}/(4m\theta^{m-1})},$$
then we clearly have
\begin{equation}
\begin{split}
\lim_{y\rightarrow\infty}\frac{\varphi'(y)}{y^{-2/(p-1)-1}}
&=\lim_{y\rightarrow\infty}\frac{e^{y^{2}/(4m\theta^{m-1})}\varphi'(y)}{e^{y^{2}/(4m\theta^{m-1})}y^{-2/(p-1)-1}}\\
&=\lim_{y\rightarrow\infty}\frac{(e^{y^{2}/(4m\theta^{m-1})}\varphi'(y))'}{(e^{y^{2}/(4m\theta^{m-1})}y^{-2/(p-1)-1})'}\\
&=\lim_{y\rightarrow\infty}\frac{1}{m\theta^{m-1}}\cdot\frac{-\left(\frac{1}{p-1}\varphi(y)+\varphi^{p}(y)\right)e^{y^{2}/(4m\theta^{m-1})}}{e^{y^{2}/(4m\theta^{m-1})}\Big(\frac{1}{2m\theta^{m-1}}y^{-2/(p-1)}-\frac{p+1}{p-1}y^{-2/(p-1)-2}\Big)}\\
&=\lim_{y\rightarrow\infty}\frac{-2}{p-1}\frac{\varphi(y)}{y^{-2/(p-1)}}\\
&=\frac{-2}{p-1}A.\\
\end{split}
\label{3.14*}
\end{equation}
Moreover, differentiating the equation in $(\ref{3.12*})$ gives
$$\Big(m\theta^{m-1}\varphi''(y)e^{y^{2}/(4m\theta^{m-1})}\Big)'=-\Big(\frac{p+1}{2p-2}+p\varphi^{p-1}(y)\Big)\varphi'(y)e^{y^{2}/(4m\theta^{m-1})}$$
then by $\lim_{y\rightarrow\infty} \varphi(y)=0,$
\begin{equation}
\begin{split}
\lim_{y\rightarrow\infty}\frac{\varphi''(y)}{y^{-2/(p-1)-2}}
&=\lim_{y\rightarrow\infty}\frac{e^{y^{2}/(4m\theta^{m-1})}\varphi''(y)}{e^{y^{2}/(4m\theta^{m-1})}y^{-2/(p-1)-2}}\\
&=\lim_{y\rightarrow\infty}\frac{(e^{y^{2}/(4m\theta^{m-1})}\varphi''(y))'}{(e^{y^{2}/(4m\theta^{m-1})}y^{-2/(p-1)-2})'}\\
&=\lim_{y\rightarrow\infty}\frac{1}{m\theta^{m-1}}\cdot\frac{-\Big(\frac{p+1}{2p-2}+p\varphi^{p-1}(y)\Big)\varphi'(y)e^{y^{2}/(4m\theta^{m-1})}}{e^{y^{2}/(4m\theta^{m-1})}\Big(\frac{1}{2m\theta^{m-1}}y^{-2/(p-1)-1}-\frac{2p}{p-1}y^{-2/(p-1)-3}\Big)}\\
&=\lim_{y\rightarrow\infty}-\frac{p+1}{p-1}\frac{\varphi'(y)}{y^{-2/(p-1)-1}}
\lim_{y\rightarrow\infty}\frac{\frac{1}{2}+\frac{p(p-1)}{p+1}\varphi^{p-1}(y)}{\frac{1}{2}-m\theta^{m-1}\frac{2p}{y^2(p-1)}}\\
&=\frac{p+1}{p-1}\cdot\frac{2}{p-1}A.\\
\end{split}
\label{3.15*}
\end{equation}
\indent Define
$$n(x,t)=t^{-1/(p-1)}\varphi(|x|/\sqrt{t}),$$
then $n(x,t)$ satisfies (\ref{3.11}) for $t>0$, $x\in\mathbb{R}$ and there exists a time $T_{1}>0$ such that for all $t>T_{1}$, $n(x,t)<\sigma$. \\
\indent
Set
$$\overline{w}(x,t)=\theta+\frac{1}{2}n(x,t).$$
For $t>T_{1}$, from $(\ref{3.12*})$ and the mean value theorem
\begin{align*}
\overline{w}_{t}-(\overline{w}^{m})_{xx}-f(\overline{w})&=\frac{1}{2}n_{t}-\frac{m}{2}\left(\theta+\frac{n}{2}\right)^{m-1}n_{xx}-\frac{m(m-1)}{4}\left(\theta+\frac{n}{2}\right)^{m-2}n_{x}^{2}-\left(\frac{1}{2}n\right)^{p}\\
&=\frac{-1}{2t^{\frac{1}{p-1}+1}}\left(\frac{1}{p-1}\varphi\Big(\frac{|x|}{\sqrt{t}}\Big)+\frac{|x|}{2\sqrt{t}}\varphi'\Big(\frac{|x|}{\sqrt{t}}\Big)+
m\theta^{m-1}\varphi''\Big(\frac{|x|}{\sqrt{t}}\Big)+\varphi^{p}\Big(\frac{|x|}{\sqrt{t}}\Big)\right)\\
&\ \ \ \ +\frac{1}{2t^{\frac{1}{p-1}+1}}\left(m\theta^{m-1}\varphi''\Big(\frac{|x|}{\sqrt{t}}\Big)-m\Big(\theta+\frac{n}{2}\Big)^{m-1}\varphi''\Big(\frac{|x|}{\sqrt{t}}\Big)
\right)\\
&\ \ \ \ +\frac{1}{2t^{\frac{1}{p-1}+1}}\left(\varphi^{p}\Big(\frac{|x|}{\sqrt{t}}\Big)-(\frac{1}{2})^{p-1}\varphi^{p}\Big(\frac{|x|}{\sqrt{t}}\Big)\right)\\
&\ \ \ \ -\frac{1}{2t^{\frac{1}{p-1}+1}}\frac{m(m-1)}{2}\Big(\theta+\frac{n}{2}\Big)^{m-2}\frac{1}{t^{\frac{1}{p-1}}}\left(\varphi'\Big(\frac{|x|}{\sqrt{t}}\Big)\right)^{2}\\
& =\frac{1}{2t^{\frac{1}{p-1}+1}}\left(\frac{1}{2}m(m-1)(\theta+\lambda)^{m-2}|\varphi''\Big(\frac{|x|}{\sqrt{t}}\Big)|\varphi\Big(\frac{|x|}{\sqrt{t}}\Big)\frac{1}{t^{\frac{1}{p-1}}}
\right)\\
&\ \ \ \ +\frac{1}{2t^{\frac{1}{p-1}+1}}\left(\varphi^{p}\Big(\frac{|x|}{\sqrt{t}}\Big)-\Big(\frac{1}{2}\Big)^{p-1}\varphi^{p}\Big(\frac{|x|}{\sqrt{t}}\Big)\right)\\
&\ \ \ \ -\frac{1}{2t^{\frac{1}{p-1}+1}}\frac{m(m-1)}{2}\Big(\theta+\frac{n}{2}\Big)^{m-2}\frac{1}{t^{\frac{1}{p-1}}}\left(\varphi'\Big(\frac{|x|}{\sqrt{t}}\Big)\right)^{2},
\end{align*}
where $\lambda\in (0,n/2)$.
%Using Mean Value Theorem, we derive
%$$\Big(m(\theta+\frac{n}{2})^{m-1}\varphi''(\frac{|x|}{\sqrt{t}})-m\theta^{m-1}\varphi''(\frac{|x|}{\sqrt{t}})\Big)
%\leq\frac{1}{2}m(m-1)(\theta+\sigma)^{m-2}|\varphi''(\frac{|x|}{\sqrt{t}})|\varphi(\frac{|x|}{\sqrt{t}})\frac{1}{t^{\frac{1}{p-1}}}.$$

Clearly, $(\ref{3.13*})$, $(\ref{3.14*})$ and $(\ref{3.15*})$ entail that $\varphi(y)\sim y^{-2/(p-1)}$,
$\varphi^{p}(y)\sim y^{-2/(p-1)-2}$,
 $\varphi'(y)\sim y^{-2/(p-1)-1}$ and
 $\varphi''(y)\sim y^{-2/(p-1)-2}$ as $y\rightarrow\infty$,
this guarantees that for sufficiently large $t$,
\begin{equation*}
\overline{w}_{t}-(\overline{w}^{m})_{xx}-f(\overline{w})=\frac{1}{2t^{\frac{1}{p-1}+1}}\left(\varphi^{p}\Big(\frac{|x|}{\sqrt{t}}\Big)-\Big(\frac{1}{2}\Big)^{p-1}\varphi^{p}\Big(\frac{|x|}{\sqrt{t}}\Big)\right)+o\left(\frac{1}{t^{\frac{1}{p-1}+1}}\right)>0
\end{equation*}
uniformly in $x$, which implies when $t$ is sufficiently large, $\overline{w}(x,t)$ is an upper solution of $(\ref{1.1})$.\\
\indent
Define
$$\underline{w}(x,t)=\theta+\frac{k}{(at+x^{2})^{1/(p-1)}},$$
direct computation yields that
\begin{equation*}
\begin{split}
\underline{w}_{t}-(\underline{w}^{m})_{xx}-f(\underline{w})&\leq\underline{w}_{t}-(\underline{w}^{m})_{xx}\\
&=\frac{-1}{p-1}\frac{k}{(at+x^{2})^{1/(p-1)+1}}(a-2m\underline{w}^{m-1})\\
&\ \ \ \ -m\underline{w}^{m-1}\frac{p}{(p-1)^{2}}\frac{4kx^{2}}{(at+x^{2})^{1/(p-1)+2}}-m(m-1)\underline{w}^{m-2}\underline{w}_{x}^{2},
\end{split}
\end{equation*}
if $a-2m\underline{w}^{m-1}>0$, we can immediately infer that $\underline{w}(x,t)$ is a lower solution of $(\ref{1.1})$. Moreover, for sufficiently small $k$, we have that $\underline{w}\leq\overline{w}$. Consequently, there exist $a>0$, $k>0$ and $T_{1}>0$ such that for $t>T_{1}$,  $\overline{w}(x,t+T_{1})$ and $\underline{w}(x,t+T_{1})$ are the upper and lower solutions of $(\ref{1.1})$, respectively. Thus there exists a symmetrically decreasing solution $w(x,t)$ of $(\ref{1.1})$ such that  $\underline{w}(x,t+T_{1})\leq w(x,t)\leq\overline{w}(x,t+T_{1})$ for $x\in \mathbb{R}$, $t>0$ and $w(\cdot,t)\rightarrow\theta$ as $t\rightarrow\infty$ locally uniformly in $C^{1}(\mathbb{R})$.\\
\indent
Note that $u(0,t)>\theta$ for all $t\geq0$. Thus there exists $T>0$ such that
$$u(0,0)>w(0,T).$$
Since $w(\cdot,t)\rightarrow\theta$ as $t\rightarrow\infty$ locally uniformly in $C^{1}(\mathbb{R})$, we can assume that $u(x,0)$ and $w(x,T)$ have exact one intersection point $y\in(0,r(t))$, which is non-degenerate.\\
\indent
We write $$v_{1}(x,t):=\frac{m}{m-1}u^{m-1}(x,t),\quad \quad v_{2}(x,t):=\frac{m}{m-1}w^{m-1}(x,t+T),$$
which evidently  satisfy the linear equation
$$v_{t}=(m-1)vv_{xx}+v_{x}^{2}+g(v),$$
where
$$g(v):=m\left(\frac{m-1}{m}v\right)^{\frac{m-2}{m-1}}f\left(\Big(\frac{m-1}{m}v\Big)^{\frac{1}{m-1}}\right).$$
We introduce
$$\eta(x,t)=v_{1}(x,t)-v_{2}(x,t)\ \ \text{for}\ \ t\geq0,\ x\in[-r(t),r(t)].$$
Then $\eta(x,t)$ satisfies
\begin{equation}
\eta_{t}=a(x,t)\eta_{xx}+b(x,t)\eta_{x}+c(x,t)\eta,\label{3.12}
\end{equation}
with
$$a(x,t):=(m-1)v_{1}(x,t),\quad b(x,t):=v_{1x}(x,t)+v_{2x}(x,t),$$
and
$$c(x,t):=(m-1)v_{2xx}(x,t)+\frac{g(v_{1}(x,t))-g(v_{2}(x,t))}{v_{1}(x,t)-v_{2}(x,t)}.$$
We claim that
$$\eta(0,t)>0\ \ \text{for all}\ \ t>0.$$
Otherwise there exists $t_{1}>0$ such that
$$\eta(0,t)>0\ \ \text{for}\ \ t\in[0,t_{1})\ \text{and}\ \eta(0,t_{1})=0.$$
Since $\eta(x,t)$ is an even function and only has two zeros $y(t)$ and $-y(t)$, by Zero number diminishing properties \cite[Lemma 3.2]{lou2024convergence}, we deduce that $y(t_{1})=0$, which implies
$$\eta(x,t)<0\ \text{for}\ t=t_{1},\ x\in [-r(t),r(t)]/\{0\}.$$
Then by the strong maximum principle,
$$\eta(x,t)<0\ \text{for}\ t>t_{1},\ x\in [-r(t),r(t)],$$
that is
$$u(x,t)<w(x,t+T)\ \ \text{for}\ t>t_{1},\ x\in[-r(t),r(t)].$$
Fix $t_{2}>t_{1}$, we can choose $\epsilon_{0}>0$ small enough such that
$$(1+\epsilon_{0})u(x,t_{2})\leq w(x,t_{2}+T)\ \ \text{for}\ x\in[-r(t_{2}),r(t_{2})].$$
By the comparison principle we have
$$u(x,t;(1+\epsilon_{0})u(x,t_{2}))\leq w(x,t+t_{2}+T)\ \ \text{for}\ \ t>0.$$
However, Theorem \ref{A} implies
$$u(x,t;(1+\epsilon_{0})u(x,t_{2}))\rightarrow1,\ \text{as}\ \ t\rightarrow\infty.$$
This contradicts the fact that $w(\cdot,t)\rightarrow\theta$ locally uniformly in $\mathbb{R}$ as $t\rightarrow\infty$. The claim is proved.\\
\indent
 As a result, we derive for all $t>0$
$$u(0,t)>w(0,t+T)>\underline{w}(0,t+T)=\theta+\frac{k}{(at+aT)^{1/(p-1)}}>\theta+C(t+T)^{-\frac{1}{p-1}}$$
with $C=k/2a^{1/(p-1)}$, we obtain the desired assertion. \hfill$\Box$

\indent
Now, we are able to estimat the lower bound of $r(t)$ for $p>3$.
\begin{yl}
Assume that $p>3$ and let $u(x,t)$ be the solution of (\ref{1.1}) with the initial datum $\widetilde{u}_{0}\in \mathfrak{X}$. Then there exists a constant $C_1>0$ such that
$$r(t)\geq 2y_{0}\sqrt{t}+C_1t^{\frac{1}{2}-\frac{1}{p-1}}\ \ \text{as} \ \ t\rightarrow\infty.$$
where $y_{0}>0$ is the zero of $\xi(y)$.
\end{yl}
\noindent\textbf{Proof.}  We write
$$\alpha(t):=\frac{a}{\theta}(T+t)^{-\frac{1}{p-1}},$$
$$\beta(t):=(T_{1}+t)^{\frac{1}{2}}(1+k(T_{1}+t)^{-\frac{1}{p-1}}),$$
%$$\beta(t)'=\frac{1}{2}(T_{1}+t)^{-\frac{1}{2}}(1+k(T_{1}+t)^{-\frac{1}{p-1}})-\frac{1}{p-1}k(T_{1}+t)^{-\frac{1}{p-1}-\frac{1}{2}},$$
and define
$$\underline{u}(x,t):=(1+\alpha(t))e(x,\beta^{2}(t))=(1+\alpha(t))\xi^{\frac{1}{m}}\left(\frac{x}{2\beta(t)}\right),$$
where $T$ is the constant determined by Lemma \ref{yl3.4} and $a\in(0,c)$, $k$ and $T_{1}$ are constants to be chosen later.

Basic calculation yields
\begin{equation}
\begin{split}
&\ \ \ \ \underline{u}_{t}(x,t)-(\underline{u}^{m})_{xx}-f(\underline{u})\\
&=\alpha(t)'\xi^{\frac{1}{m}}-
(1+\alpha(t))(\xi^{\frac{1}{m}})'\frac{x}{2\beta^{2}(t)}\beta'(t)-(1+\alpha(t))^{m}\xi''\frac{1}{4\beta^{2}(t)}-f(\underline{u})\\
&\leq -(1+\alpha(t))(\xi^{\frac{1}{m}})'\frac{x}{2\beta^{2}(t)}\beta'(t)+(1+\alpha(t))^{m}\frac{2x}{2\beta(t)}(\xi^{\frac{1}{m}})'\frac{1}{4\beta^{2}(t)}\\
&=(\xi^{\frac{1}{m}})'\frac{x}{2\beta^{2}(t)}(1+\alpha(t))\left((1+\alpha(t))^{m-1}\frac{1}{2\beta(t)}-\beta'(t)\right).
\end{split}
\end{equation}
Since $\xi'(x)<0$ and $\alpha(t),\ \beta(t)>0$, if $(1+\alpha(t))^{m-1}-2\beta'(t)\beta(t)>0$, we have
$$\underline{u}_{t}(x,t)-(\underline{u}^{m})_{xx}-f(\underline{u})\leq0.$$
In addition,
\begin{equation}
\begin{split}
&(1+\alpha(t))^{m-1}-2\beta'(t)\beta(t)\\
=&(1+\alpha(t))^{m-1}-2\left(1+k(T_{1}+t)^{-\frac{1}{p-1}}\right)\left[\frac{1}{2}(1+k(T_{1}+t)^{-\frac{1}{p-1}})-\frac{1}{p-1}k(T_{1}+t)^{-\frac{1}{p-1}}\right]\\
=&\left(1+\frac{a}{\theta}(Tt)^{-\frac{1}{p-1}}\right)^{m-1}-\left(1+k(T_{1}+t)^{-\frac{1}{p-1}}\right)\left[1+k\left(1-\frac{2}{p-1}\right)(T_{1}+t)^{-\frac{1}{p-1}}\right],\\
\end{split}
\end{equation}
for fixed $T,\ T_{1}$, choosing $k$ sufficiently small, we have $(1+\alpha(t))^{m-1}-2\beta'(t)\beta(t)>0$ for all $t>0$.\\
 \indent
 Fixing $T$, by Lemma \ref{yl3.4}, letting $a$ and $T_{1}$ $\rightarrow0$, we see $\underline{u}(x,0)$ approaches the line $x=0$ and $\underline{u}(x,0)<\theta+a$, thus for the sufficiently small $a$ and $T_{1}$, $\underline{u}(x,0)<\widetilde{u}_{0}(x)$. Consequently, $\underline{u}(x,t)$ is a lower solution to (\ref{1.1}), denote the right boundary of $\underline{u}(x,t)$ by $\underline{r}(t)$,
applying the comparison principle in \cite{aronson1982stabilization}, we conclude that
\begin{equation}
\begin{split}
r(t)&\geq \underline{r}(t)\\
&=2(T_{2}+t)^{\frac{1}{2}}\left(1+k(T_{1}+t)^{-\frac{1}{p-1}}\right)y_{0}\\
&=2y_{0}\sqrt{t}+C_1t^{\frac{1}{2}-\frac{1}{p-1}}+o(1)\ \ \ \text{as}\ \ t\rightarrow\infty.
\end{split}\label{3.15}
\end{equation}
This completes the lemma. \hfill$\Box$
\subsection{Lower bound estimate of $r(t)$ for $1<p\leq3$}
$\quad$
This subsection is devoted to establishing the lower bound of $r(t)$ for $1<p\leq 3$.
%Inspired by (\ref{3.15}), for $1<p\leq3$, it seems that $2y_{0}\sqrt{t}+k$ is the lower bound of $r(t)$.

\begin{yl}
Assume that $1<p\leq3$ and let $u(x,t)$ be the solution of $(\ref{1.1})$ with the initial datum $\widetilde{u}_{0}\in\mathfrak{X}$, Then we have
$$r(t)\geq 2y_{0}\sqrt{t}\ \ \text{for all} \ \ t>0,$$
where $y_{0}>0$ is the zero of $\xi(y)$.\label{yl3.6}
\end{yl}
\noindent\textbf{Proof.} We construct a lower solution in the following manner:
\begin{equation}
\underline{u}(x,t):=\left\{
\begin{split}
&e(x,t)=\xi^{\frac{1}{m}}\left(\frac{x}{2\sqrt{t}}\right),\quad 0\leq x\leq 2y_{0}\sqrt{t},\\
&0,\qquad\qquad\qquad\quad\qquad \ \, x\geq 2y_{0}\sqrt{t}.
\end{split}
\right.
\end{equation}
 For sufficiently small $\epsilon>0$ and $0\leq x\leq 2y_{0}\sqrt{\epsilon}$,
 $$\underline{u}(x,\epsilon)\leq \underline{u}(0,\epsilon)=\theta<u(0,0).$$
By the continuity of $u(x,0)$, we can find a small $\delta>0$ such that
$$u(x,0)>\theta\ \ \text{for}\ \ x\in[0,\delta).$$
Then if $2y_{0}\sqrt{\epsilon}<\delta$, we have $\underline{u}(x,\epsilon)<u(x,0)$ and
\begin{equation*}
\underline{u}_{t}-(\underline{u}^{m})_{xx}=e_{t}(x,t)-(e^{m})_{xx}(x,t)=0\leq f(\underline{u}).
\end{equation*}
Consequently, $\underline{u}(x,t)$ is a lower solution of (\ref{1.1}), by the comparison principle in \cite{aronson1982stabilization}, we derive
$$\underline{r}(t+\epsilon)=2y_{0}\sqrt{t+\epsilon}\leq r(t)\ \ \text{for all}\ \ t>0.$$
In view of the arbitrariness of $\epsilon$, it then follows that
$$2y_{0}\sqrt{t}\leq r(t)\ \ \text{for all}\ \ t>0.$$
The lemma is complete. \hfill$\Box$

Combining Lemma \ref{yl2.1} and Lemmas \ref{yl3.3}-\ref{yl3.6}, we obtain Theorem \ref{dl1.1}.

\end{document}